\documentclass[12pt,draftcls,onecolumn]{IEEEtran}

\IEEEoverridecommandlockouts                              

\makeatletter
\let\NAT@parse\undefined
\makeatother

\usepackage{cite}
\usepackage{amsthm}
\usepackage{amsmath,amssymb,amsfonts}
\usepackage{algorithmic}
\usepackage{graphicx}
\usepackage{textcomp}

\def \L {\mathcal{L}}
\def \a {\alpha}
\def \b {\beta}
\def \l {\lambda}

\newtheorem{example}{Example}
\newtheorem{proposition}{Proposition}
\newtheorem{theorem}{Theorem}

\newtheorem{lemma}{Lemma}
\newtheorem{corollary}{Corollary}
\newtheorem{remark}{Remark}

\begin{document}
\title{Koopman Bilinearization of Nonlinear Control Systems}
\author{Wei Zhang, \IEEEmembership{IEEE Member} and Jr-Shin Li, \IEEEmembership{Senior Member, IEEE}
\thanks{W. Zhang is with the Department of Electrical and Systems Engineering, Washington University,
		St. Louis, MO 63130, USA
        {\tt\small wei.zhang@wustl.edu}}%
\thanks{J.-S. Li is with the Department of Electrical and Systems Engineering, Washington University,
        St. Louis, MO 63130, USA
        {\tt\small jsli@wustl.edu}}%

}

\maketitle

\begin{abstract}
Koopman operators, since introduced by the French-born American mathematician Bernard Koopman in 1931, have been employed as a powerful tool for research in various scientific domains, such as ergodic theory, probability theory, geometry, and topology, with widespread applications ranging from electrical engineering and machine learning to biomedicine and healthcare. The current use of Koopman operators mainly focuses on the characterization of spectral properties of ergodic dynamical systems. In this paper, we step forward from unforced dynamical systems to control systems and establish a systematic Koopman control framework. Specifically, we rigorously derive a differential equation system governing the dynamics of the Koopman operator associated with a control system, and show that the resulting system is a bilinear system evolving on an infinite-dimensional Lie group, which directly leads to a global bilinearization of control-affine systems. Then, by integrating techniques in geometric control theory with infinite-dimensional differential geometry, this further offers a two-fold benefit to controllability analysis: the characterization of controllability for control-affine systems in terms of de Rham differential operators and the extension of the Lie algebra rank condition to systems defined on infinite-dimensional Lie groups. To demonstrate the applicability, we further adopt the established framework to develop the Koopman feedback linearization technique, which, as one of the advantages, waives the controllability requirement for the systems to be linearized by using the classical feedback linearization technique. In addition, it is worth mentioning that a distinctive feature of this work is the maximum utilization of the intrinsic algebraic and geometric properties of Koopman operators, instead of spectral methods and the ergodicity assumption of systems, which further demonstrate a significant advantage of our work and a substantial difference between the presented and existing research into Koopman operators. 
\end{abstract}


\section{Introduction}
\label{sec:introduction}

Koopman operators, named after the French-born American mathematician Bernard Osgood Koopman, was originally introduced in 1931 for the purpose of adopting Hilbert space methods, state-of-the-art techniques just coming into prominence at that time, to investigate spectral properties of Hamiltonian systems \cite{Koopman31}. Since then, the potential of Koopman operators has started to be widely recognized and extensively explored in numerous scientific domains. First and foremost, Koopman operators lie in the repertoire of the most powerful tools in ergodic theory, especially for the characterization of ergodicity, recurrence, mixing, and topological entropy of (measure-preserving) dynamical systems \cite{Walters82,Petersen83,Mane87}. These seminal research works further opened up various interdisciplinary applications of Koopman operators, notably, the study of geodesic flows on Riemannian manifolds in geometry and topology \cite{Hopf71,Katok95}, the analysis of statistical properties of Markov processes in probability theory \cite{Karatzas98,Doob91}, and the construction of Diophantine approximation in number theory \cite{Cornfeld82,Einsiedler11}. Recently, a large body of research into Koopman operators have been concentrating on the computational, data-driven, and learning aspects, particularly, for inferring temporal evolutions of dynamical systems from measurement data by using spectral-type methods, e.g., Arnoldi, vector Prony, and dynamic mode decomposition methods \cite{Rowley09,Susuki16,Raak16}. Dominant works in this line of research include spectral decompositions of fluid mechanical systems \cite{Rowley09,Budisic12,Mezic13} in mechanical engineering, stability analysis of power systems in electrical engineering \cite{Susuki11,Susuki16,Raak16}, cognitive classification and seizure detection in biomedicine and healthcare, respectively \cite{Zhang19}, and multi-modal learning and prediction in biostatistics \cite{Qian22}. 

Although Koopman operators have demonstrated their capabilities in the study of (unforced) dynamical systems, their roles in control theory for analyzing and manipulating dynamical systems with external inputs still largely remain as an unexplored avenue. To fill in this gap, in this paper, we rigorously establish a systematic Koopman control framework, particularly for tackling control tasks involving nonlinear systems, which in turn simultaneously expands the repertoire of techniques in geometric control theory and the scope of Koopman operator theory. In particular, we start from the rigorous derivation of operator differential equations governing the dynamics of the Koopman operators, referred to as Koopman systems, associated with nonlinear systems by leveraging the theory of semigroups in functional analysis. Built upon this and inspired by the lifting property of Koopman operators that model the dynamics of nonlinear systems on finite-dimensional manifolds by linear transformations on infinite-dimensional vector spaces, we define the notation of Koopman bilinearization for control-affine systems, which then leads to the characterization of controllability for such systems in terms of de Rham differential operators. Independently, we also show that the resulting bilinear Koopman systems evolve on some infinite-dimensional Lie groups, and then investigate their controllability properties by integrating techniques in geometric control theory and infinite-dimensional differential geometry. In particular, this gives rise to an extension of the Lie algebra rank condition (LARC) to systems defined on infinite-dimensional Lie groups. The established Koopman control framework is then adopted in the development of the Koopman feedback linearization technique, which not only demonstrates the applicability of this Koopman control framework but also carries out an extension of the classical feedback linearization technique, e.g., the controllability requirement for the systems to be feedback linearized is waived. 

It is worth mentioning that, in this work, the leverage of Koopman operators does not require any measure-preserving property or ergodicity of the corresponding dynamical systems, and the development of the Koopman control framework does not involve spectral methods, either. This unique feature not only highlights the major advantage of our approach but also indicates a substantial difference between the presented and existing works on Koopman operators. In addition, owing to the infinite-dimensional nature of Koopman systems and their intimacy with finite-dimensional systems, this work further sheds light on a general framework of analysis and control of infinite-dimensional systems by using finite-dimensional methods. 

The paper is organized as follows. In Section \ref{sec:Koopman}, we define Koopman operators from the perspective of dynamical systems theory, and then launch a detailed investigation into their algebraic and analytic properties, which provide the tools for the derivation of the differential equations governing the dynamics of Koopman systems in a coordinate-free way in Section \ref{sec:Koopman_dynamics}. The focus of Section \ref{sec:Koopman_control}, the main section of the paper, is on the study of Koopman systems associated with control systems. In particular, we introduce the Koopman bilinearization technique for control-affine systems, which leads to the de Rham differential operator characterization of controllability for such systems. In parallel, controllability of bilinear Koopman systems are also systematically studied, leading an extension of the LARC to systems defined on infinite-dimensional Lie groups. At last, in Section \ref{sec:Koopman_feedback_linearization}, the developed framework is adopted to extend the classical feedback linearization technique from the perspective of Koopman systems.

\section{Koopman Operators of Dynamical Systems}
\label{sec:Koopman}
In this section, we will briefly review the Koopman operator theory from the dynamical systems viewpoint. The emphasis will be placed on investigating the analytical properties of Koopman operators as linear operators for preparing the operator-theoretic analysis of control systems in the following sections. In addition, it is worth mentioning that, unlike many existing literatures on Koopman operator theory, our approach do not require any ergodicity or measure-preserving property of dynamical systems.  

\subsection{Flows of dynamical systems}
\label{sec:flow}
Given a dynamical system evolving on a finite-dimensional smooth manifold $M$, defined by the ordinary differential equation
\begin{align}
\label{eq:system}
\frac{d}{dt}x(t)=f\big(x(t)\big),
\end{align}
with $f$ a smooth vector field on $M$, we suppose that $f$ is \emph{complete}, namely, the solution of the differential equation starting from any point in $M$ exists for all $t\in\mathbb{R}$. This property holds, e.g., in the cases that $f$ is compactly supported, $M$ is compact, or $M$ is a Lie group with $f$ a left- or -right invariant vector field \cite{Lee12}. 

Under the completeness assumption, the vector field $f$ generates a smooth left action of the additive group $\mathbb{R}$ on $M$, called the \emph{(global) flow} or \emph{one-paramater group action} of $f$, given by the smooth map $\Phi:\mathbb{R}\times M\rightarrow M$ sending $(t,x)$ to the solution of the system in \eqref{eq:system} at the time $t$ starting from the initial condition $x$. As a group action, the flow $\Phi$ satisfies 
\begin{align}
\label{eq:group_law}
\Phi(t,\Phi(s,x))=\Phi(s+t,x),\quad \Phi(0,x)=x
\end{align}
for all $s,t\in\mathbb{R}$ and $x\in M$, and further gives rises to two collections of maps:
\begin{itemize}
\item $\Phi_t:M\rightarrow M$ given by $x\mapsto\Phi(t,x)$ for each $t\in\mathbb{R}$. The completeness of $f$ and uniqueness of solutions of the ordinary differential equation in \eqref{eq:system} implies that $\Phi_t$ is a diffeomorphism with $\Phi^{-1}_t=\Phi_{-t}$, and we also refer to $\Phi_t$ as the time-$t$ flow of the system. Moreover, the properties of $\Phi$ in \eqref{eq:group_law} are equivalent to the group laws: $\Phi_t\circ\Phi_s=\Phi_{s+t}$ and $\Phi_0={\rm Id}_M$, where ${\rm Id}_M$ is the identity map on $M$. In another word, the family of maps $\{\Phi_t\}_{t\in\mathbb{R}}$ is a subgroup of the diffeomorohism group of $M$.  

\item $\Phi^{x}:\mathbb{R}\rightarrow M$ given by $t\mapsto\Phi(t,x)$ for each $x\in M$. Then, $\Phi^{x}$ is exactly the integral curve of $f$ passing through $x$. Infinitesimally, this simply means that $\Phi^{x}$ satisfies the differential equation in \eqref{eq:system} for any $x\in M$ as
\begin{align*}
\frac{\partial}{\partial t}\Phi(t,x)=f(\Phi(t,x))
\end{align*}
or equivalently, in the coordinate-free representation
\begin{align}
\label{eq:flow}
\Phi_*\Big(\frac{\partial}{\partial t}\Big)=f\circ\Phi_t
\end{align}
where $\Phi_*=d\Phi:T(\mathbb{R}\times M)\rightarrow TM$ is the pushforward (differential) of $\Phi$, and $T(\mathbb{R}\times M)$ and $TM$ are the tangent bundles of $\mathbb{R}\times M$ and $M$, respectively. The fact of $\Phi_t$ being a diffeomorohism guarantees that $\Phi_*$ is globally well-defined and smooth.  
\end{itemize}

\subsection{Koopman operators}
\label{sec:Koopman_properties}
In practice, knowledge about a dynamical system may only be acquired from some observables of the system. Following the notations used in the previous section, the uniqueness of solutions of the system in \eqref{eq:system} indicates that the state variable $x(t)$ completely represents the dynamics of the system, and hence data collected from observables of the system must be in the form of functions of $x(t)$. Mathematically, this implies that observables of the system are just functions defined on the state-space of the system. To be consistent with the smoothness of the state-space manifold $M$ and the vector field $f$ governing the system dynamics, we focus on observables in $C^\infty(M)$, the space of real-valued smooth functions defined on $M$. Naturally, the flow on $M$ generated by the dynamics of the system in \eqref{eq:system} induces a flow on the space of observables $C^\infty(M)$ of the system. This then opens up the possibility of understanding the dynamics of the system by using the flow on the space of its observables, which also motivates the introduction of Koopman operators.

\subsubsection{Koopman operators and Koopman groups}
\label{sec:Koopman_algebra}
Formally, for each $t\in\mathbb{R}$, we define the \emph{Koopman operator} $U_t$ of the system as the composition operator with the time-$t$ flow of the system, that is, $U_th=h\circ\Phi_t$ for any $h\in C^\infty(M)$. Because the function composition operation is linear, $U_t$ is a linear operator on $C^\infty(M)$. Together with the fact that the composition $h\circ\Phi_t$ of the smooth functions $h$ and $\Phi_t$ is also a smooth function in $C^\infty(M)$, $U_t$ maps $C^\infty(M)$ into itself. Moreover, as $t$ varying, the group laws satisfied by the time-$t$ flow $\Phi_t$ are also inherited by the Koopman operator $U_t$, i.e., 
\begin{align}
\label{eq:Koopman_group_law}
U_0=I\quad\text{and}\quad U_{s+t}=U_sU_t
\end{align}
for any $s,t\in\mathbb{R}$, where $I$ denotes the identity map on $C^\infty(M)$. In another word, the Koopman operator $U_t$ gives rise to a flow on $C^\infty(M)$, and the ``vector field", or equivalently the dynamical system, on $C^\infty(M)$ generating this flow is of great interest and importance in this paper, which will be fully investigated from the analytical perspective in Section \ref{sec:Koopman_dynamics}. On the other hand, implied by the invertibility of $\Phi_t$ with the inverse $\Phi_{-t}$, each $U_t$ is also invertible with the inverse given by $U_{-t}$, that is, the map $h\mapsto h\circ\Phi_{-t}$. Therefore, the family of operators $\{U_t\}_{t\in\mathbb{R}}$ indeed defines a one-parameter subgroup of the group of invertible linear operators from $C^\infty(M)$ onto itself under the group operation of operator compositions, and it will be referred to as the \emph{Koopman group} associated with the dynamical system in \eqref{eq:system}.

\subsubsection{Analytical properties of Koopman operators}
\label{sec:Koopman_analysis}
To investigate the analytical properties of the Koopman operator $U_t$, it is inevitable to leverage the topology on the observable space $C^\infty(M)$. To motivate the idea, we recall that in the case of $M$ being an open subspace of the Euclidean space $\mathbb{R}^n$, the topology on $C^\infty(M)$ generated by the family of seminorms 
\begin{align*}
\|h\|_{K,\a}=\sup_{x\in K}\Big|\frac{\partial^\a}{\partial x^\a}h(x)\Big|
\end{align*}
for any multi-index $\a\in\mathbb{N}^n$ and compact subset $K\subset M$, called the $C^\infty$-topology on $C^\infty(M)$, makes $C^\infty(M)$ a locally convex topological vector space \cite{Schaefer99}. For a general finite-dimensional smooth manifold $M$, because $M$ is locally homeomorphic to an Euclidean space, the aforementioned construction of $C^\infty$-topology for $C^\infty(M)$ will be still valid if the differentiation operations in those seminorms $\|\cdot\|_{\a,K}$ are well-defined on $M$. In particular, a suitable replacement for the partial derivatives is a covariant derivative $D:C^\infty(M)\rightarrow C^\infty(T^*M)$ on $M$, where $T^*M$ denotes the cotangent bundle of $M$, and $D$ can always be globally defined in a coordinate-free way, e.g, by using the Levi-Civita connection of a Riemannian metric on $M$. 
Another consequence of $M$ being a manifold is that it is a $\sigma$-compact locally compact Hausdorff space \cite{Lee12}, and hence $M$ has an open cover consisting of an increasing sequence of precompact open sets $\{U_i\}_{i\in\mathbb{N}}$ such that $\overline{U}_i\subset U_{i+1}$ for all $i\in\mathbb{N}$, where $\overline{U}_m$ denotes the closure of $U_m$ (with respect to the topology of $M$) \cite{Munkres00}. As a result, every compact subset of $M$ is contained in at least one of such open set $U_m$. Integrating these observations, the $C^\infty$-topology on $C^\infty(M)$ is generated by the countable family of seminorms
\begin{align}
\label{eq:C^infty_topology}
\|h\|_{i,j}=\sup_{x\in \overline{U}_i}\Big|D^jh(x)\Big|,
\end{align}
where $D^j:C^\infty(M)\rightarrow C^\infty(\otimes^{j}T^*M)$, playing the role of the $j^{\rm th}$ power (composition with itself) of $D$, is the covariant derivative on the sections of the tensor bundle $\otimes^jT^*M$ induced by $D$, and $\otimes^jT^*M$ denotes the $j$ times tensor of $T^*M$. In another word, $C^\infty(M)$ is a Frech\'{e}t space. Note that although different choices of the open cover $\{U_i\}_{i\in\mathbb{N}}$ and covariant derivative $D$ result in different families of seminorms, but they generate the same $C^\infty$-topology on $C^\infty(M)$. 

From the analytical viewpoint, a major use of the topology on a vector space is to evaluate continuity, equivalently boundedness, of linear operators. For example, in the case of the $C^\infty$-topology on $C^\infty(M)$, the covariant derivative $D^j$, which is also an order $j$ partial differential operator on $M$, is bounded linear operator as $\|D^jh\|_{i,k}=\|h\|_{i,k+j}$ for all $i,j,k\in\mathbb{N}$. This lays the foundation for the Koopman linearization method proposed in the following sections, where we use the Lie derivative as the differential operator $D$. The task now is to show that the Koopman operator $U_t$ is also bounded. 

\begin{proposition}
\label{prop:Koopman_bounded}
Given a smooth dynamical system defined on a smooth manifold $M$ as in \eqref{eq:system}, the associated Koopman operator $U_t:C^\infty(M)\rightarrow C^\infty(M)$ is a bounded linear operator under the $C^\infty$-topology on $C^\infty(M)$.  
\end{proposition}
\begin{proof}
To show the boundedness, equivalently, continuity of $U_t$ under the topology generated by the family of seminorms in \eqref{eq:C^infty_topology}, it suffices to prove that for any $i,j\in\mathbb{N}$, there exist $i_1,\dots,i_n,j_1,\dots,j_n\in\mathbb{N}$ and $C>0$ such that $\|U_th\|_{i.j}\leq C\sup_{i_k,j_k}\|h\|_{i_k,j_k}$ \cite{Schaefer99}.

The chain rule for covariant derivatives yields
\begin{align*}
\|U_th\|_{i,j}&=\sup_{x\in\overline{U}_i}\big|D^jh(\Phi_t(x))\big|\\
&=\sup_{x\in\overline{U}_i}\big|\sum_{k=0}^jD^kh(\Phi_t(x)) p_k(\Phi_t(x),\dots,D^k\Phi_t(x))\big|
\end{align*}
for some polynomial functions $p_k$ in $k$ variables, $k=0,\dots,j$. Let $C_k=\sup_{x\in\overline{U}_i}p_k(\Phi_t(x),\dots,D^k\Phi_t(x))$, then it leads to 
\begin{align*}
\|U_th\|_{i,j}&\leq\sum_{k=0}^j C_k\sup_{x\in\overline{U}_i}\big|D^kh(\Phi_t(x))\big|\\
&=\sum_{k=0}^j C_k\sup_{y\in\Phi(\overline{U}_i)}\big|D^kh(y)\big|.
\end{align*}
Because $\Phi_t$ is smooth and $\overline{U_i}$ is compact, $\Phi_i(\overline{U_i})$ is also compact and contained in some $U_l$, which leads to the desired result
\begin{align*}
\|U_th\|_{i,j}\leq &\sum_{k=0}^jC_k\sup_{x\in\overline{U}_l}\big|D^kh(x)\big|=\sum_{k=0}^jC_k\sup_{x\in\overline{U}_l}\|h(x)\|_{l,k}\\
& \leq C\sup_{k=0,\dots,j}\|h\|_{l,k}
\end{align*}
by setting $C=j\max C_k$.
\end{proof}

Recall that all bounded linear transformations on $C^\infty(M)$ form a group ${\rm Aut}(C^\infty(M))$, called the \emph{automorphism group} of $C^\infty(M)$. Then, Proposition \ref{prop:Koopman_bounded} can also be equivalently presented in algebraic terminologies as follows.

\begin{corollary}
\label{cor:Koopman_bounded}
The Koopman group $\{U_t\}_{t\in\mathbb{R}}$ is a one parameter subgroup of the automorphism group ${\rm Aut}(C^\infty(M))$ of $C^\infty(M)$.
\end{corollary}

This algebraic characterization of Koopman operators in Corollary \ref{cor:Koopman_bounded} opens up the possibility to associate infinitesimal generators, conceptually derivatives with respect to the time variable $t$, to Koopman operators, which then naturally govern the dynamics of ordinary differential equation systems defined on ${\rm Aut}(C^\infty(M))$. The focus of the next section will be on the study of these Koopman operator-induced systems, with the emphasis on the relationship between nonlinear systems on $M$ and their Koopman systems on ${\rm Aut}(C^\infty(M))$.

\section{Koopman Linearization of Nonlinear Dynamical Systems}
\label{sec:Koopman_dynamics}

Koopman operators are defined to be associated with dynamical systems, and this immediately brings to the generic question of how the dynamics of a system reflects on the associated Koopman operator. To gain some insight into this question, taking the system in \eqref{eq:system} as an example, we first observe from \eqref{eq:flow} that the vector field $f$ governing the system dynamics can be recovered from the flow $\Phi$ by taking the time derivative. Together with the fact that the Koopman operator $U_t$ associated with the system is the composition operator with respect to the time-$t$ flow $\Phi_t$, this observation gives a clue about the intimacy between the system dynamics and the time derivative of $U_t$, which will be fully explored, mainly from the control-theoretical perspective, in this section. To be more specific, we will derivative a differential equation system governing the dynamics of the Koopman operator $U_t$, towards the goal of investigating  control-related properties of the system by using the associated Koopman system. 

\subsection{Infinitesimal generators of Koopman groups}
\label{sec:Koopman_system}

As motivated below Corollary \ref{cor:Koopman_bounded}, the differential equation system governing the dynamics of the Koopman operator $U_t$ can be obtained from the infinitesimal generator of the Koopman group $\{U_t\}_{t\in\mathbb{R}}$. To guarantee the existence of the infinitesimal generator of $\{U_t\}_{t\in\mathbb{R}}$, it is required that $\{U_t\}_{t\in\mathbb{R}}$ is strongly continuous, i.e., $U_t\rightarrow U_0$ as $t\rightarrow0$ in the strong operator topology, equivalently, $U_th\rightarrow U_0h=h$ as $t\rightarrow0$ for all $h\in C^\infty(M)$ in the $C^\infty$-topology \cite{Yosida95}.

\begin{lemma}
\label{lem:Koopman_continuous}
For any dynamical system defined on a smooth manifold $M$ as in \eqref{eq:system}, the associated Koopman group $\{U_t\}_{t\in\mathbb{R}}$ is a strongly continuous one parameter group of continuous linear operators on $C^\infty(M)$.
\end{lemma}
\begin{proof}
Let $\Phi:\mathbb{R}\times M\rightarrow M$ and $U_t$ denote the flow and Koopman operator of the system, respectively,  and pick any $h\in C^\infty(M)$. Because the $C^\infty$-topology on $C^\infty(M)$ is generated by the family of seminorms $\|\cdot\|_{i,j}$ on $M$, $U_th\rightarrow h$ is equivalent to $\|U_th-h\|_{i,j}\rightarrow0$ for all $i,j\in\mathbb{N}$. To prove the seminorm convergence, by the definition of $\|\cdot\|_{i,j}$, we have
\begin{align}
\lim_{t\rightarrow0}\|U_th- h\|_{i,j}&=\lim_{t\rightarrow0}\sup_{x\in\overline{U}_i}\big|D^j\big(h(\Phi_t(x))-h(x)\big)\big|\nonumber\\
&\leq\sup_{x\in\overline{U}_i}\big|\lim_{t\rightarrow0}D^j\big(h(\Phi_t(x))-h(x)\big)\big|\nonumber\\
&\leq\sup_{x\in\overline{U}_i}\big|D^j\big(h(\lim_{t\rightarrow0}\Phi_t(x))-h(x)\big)\big|, \label{eq:Koopman_continuous}
\end{align}
where the second and third inequalities follow from the lower semicontinuity of the supremum function and uniform continuity of $D^jh$ on the compact set $\overline{U}_i$ as a section of $\otimes^jT^*M$, respectively. Next, in \eqref{eq:Koopman_continuous}, as a function of $t$, $\Phi_t(x)=\Phi(t,x)$ is also continuous, and hence $\lim_{t\rightarrow}\Phi_t(x)=\Phi(0,x)=x$ holds for all $x\in\overline{U}_j$, giving the desired convergence $\lim_{t\rightarrow0}\|U_th- h\|_{i,j}=0$.
\end{proof}

Using more elementary terminology, Lemma \ref{lem:Koopman_continuous} simply means that the operator-valued function $\mathbb{R}\rightarrow{\rm Aut}(C^\infty(M))$ given by $t\mapsto U_t$ is continuous. Therefore, similar to the basic calculus, it is possible to take the derivative of $U_t$ with respect to $t$, that is, conceptually,
\begin{align}
\frac{d}{dt}\Big|_{t=t_0}U_t&=\lim_{t\rightarrow t_0}\frac{U_t-U_{t_0}}{t-t_0}=\lim_{t\rightarrow t_0}\frac{U_{t_0}(U_{t-t_0}-I)}{t-t_0}\nonumber\\
&=U_{t_0}\lim_{s\rightarrow 0}\frac{U_{s}-I}{s}=U_{t_0}\frac{d}{dt}\Big|_{t=0}U_t, \label{eq:Koopman_group}
\end{align}
where the second equality follows from the group law for the Koopman operator. Moreover, this also implies that the derivative of $U_t$ at any time $t$ is completely determined by its derivative at time 0, which is known as the \emph{infinitesimal generator} of the Koopman group $\{U_t\}_{t\in\mathbb{R}}$. Technically, because the continuity of $\{U_t\}_{t\in\mathbb{R}}$ holds in the strong operator topology, it is required that the derivative is valid in the same topology, meaning, the limit 
\begin{align}
\label{eq:Koopman_generator}
\frac{d}{dt}\Big|_{t=0}U_th=\lim_{t\rightarrow 0}\frac{U_{t}h-h}{t}
\end{align}
converges for all $h\in C^\infty(M)$ in the $C^\infty$-topology.

\begin{theorem}
\label{thm:Koopman_generator}
The infinitesimal generator of the Koopman group $\{U_t\}_{t\in\mathbb{R}}$ associated with the dynamical system $\frac{d}{dt}x(t)=f(x(t))$ evolving on $M$ is the Lie derivative $\mathcal{L}_f$ with respect to the vector field $f\in C^\infty(TM)$.
\end{theorem}
\begin{proof}
Pick any $h\in C^\infty(M)$, by considering $h$ as a 0-tensor field on $M$, $U_th=h\circ\Phi_t=\Phi^*_th$ is just the pull back of $h$ by the smooth function $\Phi_t$ so that
\begin{align}
\label{eq:Koopman_Lie}
\mathcal{L}_fh=\lim_{t\rightarrow0}\frac{\Phi_t^*h-h}{t}=\lim_{t\rightarrow0}\frac{U_th-h}{t}=\frac{d}{dt}\Big|_{t=0}U_th
\end{align}
following from the definition of Lie derivatives. The fact $\mathcal{L}_fh\in C^\infty(M)$ then implies that the the limit in \eqref{eq:Koopman_Lie} holds in the $C^\infty$-topology on $C^\infty(M)$, concluding the proof. 
\end{proof}

In general, in a topological vector space that is not a Banach space, e.g. the Frech\'{e}t space $C^\infty(M)$ concerned in this paper, equicontinuity of a semigroup of operators in the time variable $t$ is usually a requirement, in addition to strong continuity, to guarantee the infinitesimal generator to be densely defined on its domain \cite{Yosida95}. Unfortunately, Koopman groups generally do not satisfy the equicontintuity condition, for if so, then given any seminorm $\|\cdot\|_{i,j}$, there is a seminorm $\|\cdot\|_{k,l}$ such that $\|U_th\|_{i,j}\leq\|h\|_{k,l}$ for all $t\in\mathbb{R}$ and $h\in C^\infty(M)$. For example, consider the system $\frac{d}{dt}x(t)=1$ defined on $\mathbb{R}$, whose flow is given by $\Phi(t,x)=x+t$ so that $U_th(x)=h(x+t)$ for all $h\in C^\infty(\mathbb{R})$. However, in the case of $h(x)=e^x$, $U_th(x)=e^{x+t}$ escapes every compact subset of $\mathbb{R}$ when $t\rightarrow\infty$, and hence $U_t$ fails to be equicontinuous. Even in such a case, Theorem \ref{thm:Koopman_generator} implies that the infinitesimal generator of the Koopman group is still defined on the whole space $C^\infty(M)$, not only a dense subset of $C^\infty(M)$, since the Lie derivative is always well-defined for all smooth function on $M$, which definitely exhibits one of the most outstanding features of Koopman operators. 

\subsection{Coordinate-free Koopman linearization}
\label{sec:Koopman_dynamics_coordinates}

A recapitulation of the derivation in \eqref{eq:Koopman_group} that the derivative of the Koopman operator $U_t$ at any time is completely determined by the infinitesimal generator $\mathcal{L}_f$ in a linear way sheds light on the construction of a linear differential equation system on ${\rm Aut}(C^\infty(M))$ to describe the dynamics of the Koopman operator. 

\begin{theorem}
\label{thm:Koopman_system}
Given a dynamical system defined on a smooth manifold $M$ as in \eqref{eq:system}, that is, 
$$\frac{d}{dt}x(t)=f(x(t)),$$
 the Koopman operator $U_t$ associated with the system satisfies a linear differential equation system on ${\rm Aut}(C^\infty(M))$, given by,
\begin{align}
\label{eq:Koopman_system}
\frac{d}{dt}U_t=U_t\mathcal{L}_f=\mathcal{L}_fU_t.
\end{align}
\end{theorem}
\vspace{0.2cm}
\begin{proof}
Pick any $h\in C^\infty(M)$, the differential equation in \eqref{eq:Koopman_system} follows from computing the derivative of $U_th$ with respect to $t$ in two ways. At first, the chain rule yields
\begin{align*}
\frac{d}{dt}U_th&=\frac{d}{dt}(h\circ\Phi_t)=d(h\circ \Phi_t)\Big(\frac{d}{dt}\Big)=dh\circ d\Phi\Big(\frac{\partial}{\partial t}\Big)\\
&=dh(f)\circ\Phi_t=\mathcal{L}_fh\circ\Phi_t=U_t\mathcal{L}_fh,
\end{align*}
where we use the definition of the flow in \eqref{eq:flow} and the fact of the Lie derivative $\mathcal{L}_f$ acting on the 0-tensor, i.e., smooth function, $h$ as the directional derivative $\mathcal{L}_fh=dh(f)$ \cite{Lee12}. 

Alternatively, by using the antisymmetry of Lie derivative actions on vector fields as $\mathcal{L}_ff=[f,f]=0$, where $[\cdot,\cdot]$ denotes the Lie bracket of vector fields, we have $d\Phi_t(f)=f\circ\Phi_t$, and then it follows
\begin{align*}
&\mathcal{L}_fU_th=d(U_th)(f)=d(h\circ\Phi_t)(f)\\
&=dh\circ d\Phi_t(f)=dh(f)\circ\Phi_t=U_t\mathcal{L}_fh
\end{align*}
as desired.
\end{proof}

As indicated by the proof of Theorem \ref{thm:Koopman_system}, the system in \eqref{eq:Koopman_system} governing the dynamics of the Koopman operator holds in the strong operator topology induced by the $C^\infty$-topology on $C^\infty$ and we refer to this system as the \emph{Koopman system} associated with the system in \eqref{eq:system}. Although proved computationally, the commutativity of $U_t$ and $\mathcal{L}_f$ in the Koopman system in \eqref{eq:Koopman_system} also follows abstractly from the general facts in functional analysis that the infinitesimal generator of a strongly continuous semigroup of everywhere defined operators is always densely defined and commutes with every operator in the semigroup \cite{Yosida95}.  

\begin{remark}[Coordinate-free global Koopman linearization]
\label{rmk:Koopman_linearization}
Notice that because the Lie derivative $\mathcal{L}_f$ is a linear operator, to be more specific, a first-order partial differential operator, the Koopman system in \eqref{eq:Koopman_system} is always a linear system regardless of the nonlinearity of the associated system in \eqref{eq:system}. Together with that the representation of the Koopman system does not involve the coordinates of $M$, the linearization procedure from the nonlinear system to the associated Koopman system must hold globally. As a summary, Koopman systems give rise to coordinate-free global linearization of nonlinear system. 
\end{remark}

To gain some intuition into the mechanism hidden behind the Koopman system, we recall from linear systems theory that, for any time-invariant linear system defined on $\mathbb{R}^n$, say $\frac{d}{dt}x(t)=Ax(t)$ with $x(t)\in\mathbb{R}^n$ and $A\in\mathbb{R}^{n\times n}$, the transition matrix $e^{tA}$ of the system satisfies \cite{Brockett71}:
\begin{itemize}
\item The group laws: $e^{tA}|_{t=0}=I$, the $n$-by-$n$ identity matrix, and $e^{(s+t)A}=e^{sA}e^{tA}$ for any $s,t\in\mathbb{R}$.
\item The matrix differential equation 
\begin{align}
\label{eq:transition_matrix}
\frac{d}{dt}e^{tA}=Ae^{tA}=e^{tA}A.
\end{align}
\end{itemize}
In addition, as a fundamental solution, any solution of the system can be generated by the transition matrix as $x(t)=e^{tA}x_0$ for some $x_0\in\mathbb{R}^n$, equivalently, 
\begin{align}
\label{eq:transition_matrix_system}
\frac{d}{dt}e^{tA}x_0=Ae^{tA}x_0,
\end{align}
and $e^{tA}$ admits the power series expansion for the exponential function as
$$e^{tA}=\sum_{k=0}^\infty\frac{t^k}{k!}A^k$$ 
with the convergence radius $\infty$, i.e., the series converges for any $A\in\mathbb{R}^{n\times n}$  \cite{Brockett71}.

A comparison between the properties satisfied by the Koopman operator $U_t$ and the transition matrix $e^{tA}$ reviewed above, particularly the operator differential equation in \eqref{eq:Koopman_system} and the matrix differential equation in \eqref{eq:transition_matrix}, immediately gives an interpretation of $U_t$ from the perspective of linear systems theory: the Koopman operator $U_t$ is the fundamental solution, playing the role of the ``transition matrix" as for a finite-dimensional time-invariant linear system, of the infinite-dimensional linear system on $C^\infty(M)$
\begin{align}
\label{eq:Koopman_observable}
\frac{d}{dt}(U_th)=\mathcal{L}_f(U_th),
\end{align}
which can be considered as the infinite-dimensional analogue of the system in \eqref{eq:transition_matrix_system}. Then, similarly, the power series expansion
$$U_t=\exp(t\mathcal{L}_f)=\sum_{k=0}^\infty\frac{t^k}{k!}\mathcal{L}_f^k$$ 
also holds and converges in the strong operator topology with the convergence radius $\infty$ so that any solution of the system in \eqref{eq:Koopman_observable} admits the representation $U_th=\exp(t\mathcal{L}_f)h$ for some $h\in C^\infty(M)$, which in turn gives rise to a functional calculus definition of the Koopman operator $U_t$ as the exponential of the Lie derivative operator $\mathcal{L}_f$. 

On the other hand, by using the coordinates on $M$, the measurement data of the system generated by the observable $h\in C^\infty(M)$ is given by $y(t)=h(x(t))$. The data can be further represented in terms of the Koopman operator $U_t$ as $y(t)=h(\Phi_t(x_0))=U_th(x(0))$, and hence the equation in \eqref{eq:Koopman_observable} exactly recovers the dynamics of the data in the form of the differential equation
\begin{align}
\label{eq:data_dynamics}
\frac{d}{dt}y(t)=\mathcal{L}_fh(x(t))=dh(f)(x(t)),
\end{align}
where we use the fact that the Lie derivative $\mathcal{L}_f$ acts on the 0-tensor $h\in C^\infty(M)$ as the directional derivative along the direction $f$. In the special case of $M=\mathbb{R}^n$, the vector field $f$ can be identified with an $n$-tuple of smooth functions on $\mathbb{R}^n$, and so is $dh$, regarded as the gradient $\nabla h$ of $h$. Then, the equation in \eqref{eq:data_dynamics} reduces to the form of the chain rule
$$\frac{d}{dt}y(t)=\nabla h\cdot f(x(t)),$$
with ``$\cdot$" denoting the Euclidean inner product on $\mathbb{R}^n$, which coincides with the system governing the dynamics of the output of the system in \eqref{eq:system} derived by using the techniques in the classical nonlinear systems theory \cite{Khalil01}. In another word, the system in \eqref{eq:Koopman_observable}, equivalently, the action of the Koopman system in \eqref{eq:Koopman_system} on the space of observables $C^\infty(M)$, basically gives rise to a dynamical system theoretic coordinate-free representation of the chain rule for differentiation.

\section{Geometric Control of Koopman Systems on Lie Groups}
\label{sec:Koopman_control}

After the detailed investigation on Koopman systems in Section \ref{sec:Koopman_dynamics}, we will move one step further in this section to analyze the dynamic behavior of Koopman systems driven by control inputs. In particular, the focus will be placed on the relationship between control-affine systems and their associated Koopman control systems in terms of control-related properties, and this will lead to an operator-theoretic controllability condition for control-affine systems. However, the techniques to be developed are not constrained to examine controllability. To demonstrate the generalizability, we will establish a systematic Koopman feedback linearization framework parallel to the classical feedback linearization method by leveraging the linearity of Koopman systems. 

\subsection{Koopman operators associated with control systems}
\label{sec:Koopman_control}

A control system evolving on a smooth manifold $M$ has the form 
\begin{align}
\label{eq:control_system}
\frac{d}{dt}x(t)=F(t,x(t),u(t)),
\end{align}
where $u$ is the control input, generally assumed to take values in $\mathbb{R}^m$.  In particular, for any given the control input $u$, $F$ is a time-dependent smooth vector field on $M$. As discussed in Section \ref{sec:Koopman}, the Koopman operator associated with a dynamical system needs to be defined through the flow of the system. However, for a system whose dynamics is governed by a time-dependent vector as the system in \eqref{eq:control_system}, it might not generate a well-defined flow, even locally, as defined in Section \ref{sec:flow}, because different integral curves starting at the same point but at different times might follow different paths. To avoid this technical difficulty, we always set the starting time of the system to be 0 and treat $u$ as intermediate variables. Moreover, for any choice of the control input $u$, we also assume the vector field $F$ to be complete so that the system generates a global flow. 

Under the aforementioned technical setups, we can define the Koopman operator associated with the system in \eqref{eq:control_affine} in a similar way as before. To this end, let $\Phi:\mathbb{R}\times M\rightarrow M$ denote the flow of the system driven by some control inputs $u_i$, $i=1,\dots,m$, then $\Phi(t,x)\in M$ is the state of the system at time $t$ starting from $x\in M$ at time 0. The Koopman operator associated with the system is defined similarly as $U_t:C^\infty(M)\rightarrow C^\infty(M)$ given by $h\mapsto h\circ\Phi_t$ for any observable $h\in C^\infty(M)$ of the system. Note that, in this case, although not explicitly expressed in the formulas, the Koopman operator $U_t$, as well as the flow $\Phi_t$, contains the control inputs $u_i$, $i=1,\dots,m$ applied to the system as the intermediate variables. Moreover, for any choice of the control inputs, all of the properties of Koopman operators discussed in Section \ref{sec:Koopman} still hold, which are summarized in the following proposition.

\begin{proposition}
\label{prop:Koopman_control_group}
Given a control system evolving on a smooth manifold $M$ as in \eqref{eq:control_system}, the associated Koopman operator $U_t:C^\infty(M)\rightarrow C^\infty(M)$ satisfies the following:
\begin{enumerate}
\item The Koopman operator $U_t$ is continuous in the $C^\infty$-topology on $C^\infty(M)$.
\item The Koopman group $\{U_t\}_{t\in\mathbb{R}}$ is a strongly continuous one parameter subgroup of ${\rm Aut}(C^\infty(M))$.
\item The Koopman operator $U_t$ satisfies the differential equation system on ${\rm Aut}(C^\infty(M))$, given by,
\begin{align}
\label{eq:Koopman_system_control}
\frac{d}{dt}U_t=\mathcal{L}_FU_t=U_t\mathcal{L}_F
\end{align}
with the initial condition $U_0=I$, the identity of the group ${\rm Aut}(C^\infty(M))$.
\end{enumerate}
\end{proposition}
\begin{proof}
Driven by any control inputs $u_i$, $i=1,\dots,m$, $\Phi(\cdot,x):\mathbb{R}\rightarrow M$ gives the solution of the system in \eqref{eq:control_affine} passing through the point $x\in M$ at time 0, and hence satisfies the groups laws as in \eqref{eq:group_law} following from the uniqueness of solutions for ordinary differential equations. As a result, $\{U_t\}_{t\in\mathbb{R}}$ is a one parameter group of linear operators from $C^\infty(M)$ to itself. The continuity of $U_t$, strong continuity of the group $\{U_t\}_{t\in\mathbb{R}}$, and the derivation of the system in \eqref{eq:Koopman_system_control} follow from the same proofs as Proposition \ref{prop:Koopman_bounded}, Lemma \ref{lem:Koopman_continuous}, and Theorem \ref{thm:Koopman_system}, respectively. 
\end{proof}

Although the form of the Koopman system in \eqref{eq:Koopman_system_control} associated with a control system appears to be identical to the one in \eqref{eq:Koopman_system} associated with a dynamical system without control input, these two Koopman systems are distinct from several aspects. For example, in the control system case as in \eqref{eq:Koopman_system_control}, we particularly specify the initial condition $U_0=I$ to emphasize that the starting time can only be chosen to be 0, which is necessary to guarantee the Koopman system in \eqref{eq:Koopman_system_control} to be well-defined as discussed in the beginning of Section \ref{sec:Koopman_control}. More importantly, the Koopman system in \eqref{eq:Koopman_system_control} actually inherits the control inputs from the vector field $F$ in the Lie derivative $\mathcal{L}_F$, and hence is a control system defined on ${\rm Aut}(C^\infty(M))$ sharing the same control inputs with the associated control system in \eqref{eq:control_system} defined on $M$. Consequently, it is natural to ask for controllability of the Koopman system in \eqref{eq:Koopman_system_control}, and particularly its relationship with the associated systems in \eqref{eq:control_system} in terms of control-related properties, which will be fully explored by using control-affine systems in the next section.

\subsection{Gelfand duality between bilinear Koopman and control-affine systems}
Nonlinear systems in the control-afffine form are natural models for describing the dynamics of control systems arising from numerous practical and scientific applications in diverse fields, ranging from quantum physics and robotics to neuroscience. In the most general form, a (smooth) control-affine system defined on a smooth manifold $M$ can be represented as
\begin{align}
\label{eq:control_affine}
\frac{d}{dt}x(t)=f(x(t))+\sum_{i=1}^mu_i(t)g(x(t)),
\end{align}
where $f$ and $g_i$, $i=1,\dots,m$ are smooth vector fields on $M$, and $u_i:\mathbb{R}\rightarrow\mathbb{R}$ are the control inputs. The control-linear form of the system in \eqref{eq:control_affine} leads to a particularly neat representation of the associated Koopman system.

\begin{theorem}[Koopman bilinearization]
\label{thm:Koopman_bilinearization}
Given a control-affine system defined on a smooth manifold $M$ as in \eqref{eq:control_affine}, the associated Koopman system is a bilinear system defined on ${\rm Aut}(C^\infty(M))$, given by,
\begin{align}
\frac{d}{dt}U_t&=U_t\mathcal{L}_f+\sum_{i=1}^mu_i(t)U_t\mathcal{L}_{g_i}\nonumber\\
&=\mathcal{L}_fU_t+\sum_{i=1}^mu_i(t)\mathcal{L}_{g_i}U_t. \label{eq:Koopman_control_affine}
\end{align}
\end{theorem}
\vspace{0.2cm}
\begin{proof}
The result directly follows from the property 3) in Proposition \ref{prop:Koopman_control_group} by choosing $F=f+\sum_{i=1}^mu_i{g_i}$ and the linearity of the Lie derivative and Koopman operator.  
\end{proof}

Similar to the systems in \eqref{eq:Koopman_system} and \eqref{eq:Koopman_system_control}, the Koopman system in \eqref{eq:Koopman_control_affine} holds in the sense of the strong operator topology, i.e., 
\begin{align*}
\frac{d}{dt}U_th&=U_t\mathcal{L}_fh+\sum_{i=1}^mu_i(t)U_t\mathcal{L}_{g_i}h\\
&=\mathcal{L}_fU_th	+\sum_{i=1}^mu_i(t)\mathcal{L}_{g_i}U_th,
\end{align*}
for any $h\in C^\infty(M)$, because the infinitesimal generator $\frac{d}{dt}U_t$ is derived in this topology as shown in Section \ref{sec:Koopman_system}.

\begin{remark}
\label{rmk:commutative}
The Koopman system in \eqref{eq:Koopman_control_affine} requires to be interpreted with great caution: the commutativity between the Koopman operator $U_t$ and the Lie derivative operator $\mathcal{L}_{f+\sum_{i=1}^mu_ig}$ by no means implies the commutativity between $U_t$ and $\mathcal{L}_f$ as well as $U_t$ and $\mathcal{L}_{g_i}$ for each $i=1,\dots,m$. Taking $\mathcal{L}_f$ for instance, as indicated by the proof of Theorem \ref{thm:Koopman_system}, because $\mathcal{L}_{f+\sum_{i=1}^mu_ig}$ is the infinitesimal generator of $U_t$, equivalently the vector field $f+\sum_{i=1}^mu_ig$ generates the flow of the system, the commutativity between $U_t$ and $\mathcal{L}_f$ requires $\mathcal{L}_f(f+\sum_{i=1}^mu_ig_i)=\sum_{i=1}^mu_i\mathcal{L}_fg_i=0$, for which a sufficient condition is $\mathcal{L}_fg_i=[f,g_i]=0$ for all $i=1,\dots,m$. This in turn gives rise to a Koopman operator-theoretic characterization of the general fact in differential geometry that two smooth vector fields on a smooth manifold is commutative, that is, their Lie bracket vanishes, if and only if one of them is invariant under the flow of the other one \cite{Lee12}. 
\end{remark}

Remark \ref{rmk:commutative} further sheds light on the intimacy between Koopman operators and Lie brackets of vector fields. It is well-known in geometric control theory that, for a control affine system as in \eqref{eq:control_affine}, the iterative Lie brackets of, or more specifically, the Lie algebra generated by, the drift and control vector fields determine the controllability of the system \cite{Brockett76,Hermes78,Isidori95,Jurdjevic96,Brockett14}. This strongly inspires the study of controllability of control-affine systems through the associated Koopman systems and vice versa. To this end, it is required to investigate Koopman systems further from the perspective of geometric control theory.

\subsubsection{Koopman systems on Lie groups}

A remarkable feature of Koopman systems associated with control-affine systems is the bilinearity as proved in Theorem \ref{thm:Koopman_bilinearization}. Right or left invariant control systems defined on Lie groups constitute a large class of bilinear systems, and have been widely studied in geometric control theory \cite{Brockett72,Jurdjevic72,Brockett73b}. Fortunately, the Koopman system associated with a control-affine system as in \eqref{eq:Koopman_control_affine} falls into this category of bilinear systems, but it is generally an infinite-dimensional system, which presents a challenge to our study. In particular, the challenge mainly comes from the infinite-dimension of the state-space Lie group ${\rm Aut}(C^\infty(M))$ of the Koopman system in \eqref{eq:Koopman_control_affine}. 

To avoid some technical topological complications for the purpose of illuminating the main idea, we impose the compactness condition that that the state-space manifold $M$ of the control-affine system in \eqref{eq:control_affine} is closed, that is, compact, connected, and without boundary. As a result, any smooth vector field on $M$ are complete and hence generates a global flow \cite{Lee12}. Moreover, the family of seminorms in \eqref{eq:C^infty_topology} generating the $C^\infty$-topology on $C^\infty(M)$ are reduced to the form
$$\|h\|_k=\sup_{x\in M}|D^kh(x)|,$$
which then coincides with the Whitney $C^\infty$-topology so that pointwise addition and multiplication of smooth functions are continuous  \cite{Golubitsky73}. This further leads to the continuity of the map $\mathbb{R}\rightarrow{\rm Diff}(M)$ given by $t\mapsto\Phi_t$ for any flow $\Phi$ on $M$ with $M$ denoting the group of diffeomorphisms from $M$ onto $M$ \cite{Golubitsky73}. The use of these results is to identify the Lie algebra of ${\rm Aut}(C^\infty(M))$, for analyzing controllability of the Koopman system in \eqref{eq:Koopman_control_affine} defined on ${\rm Aut}(C^\infty(M))$.

\begin{lemma}
\label{lem:Lie_algebra}
Let $M$ be a closed manifold, then ${\rm Aut}(C^\infty(M))$ is a Lie group with the Lie algebra consisting of Lie derivatives with respect to smooth vector fields on $M$ under the Whitney $C^\infty$-topology on $C^\infty(M)$.
\end{lemma}
\begin{proof}
The continuity of pointwise addition and multiplication implies that $C^\infty(M)$ is a commutative seminormed $C^*$-algebra. Then, the application of Gelfand–Naimark theorem implies that ${\rm Aut}(C^\infty(M))$ is isomorphic to ${\rm Diff}(M)$  \cite{Palmer01}. To identify the Lie algebra of ${\rm Diff}(M)$, it is equivalent to identify the tangent space at the identity map ${\rm id}_M$. To this end, we choose a Riemannian metric on $M$. Because $M$ is compact, the Riemannian exponential map $\exp:TM\rightarrow M$ is globally well-defined and a local homeomorphism following from $d\exp_p(0)={\rm id}_{T_pM}$, giving the normal coordinates of $M$ \cite{Carmo92}. Therefore, by varying $p$ in $M$, we obtain that $C^{\infty}(M,TM)$, the space of smooth sections of the tangent bundle $TM$ of $M$, is tangent to $C^\infty(M,M)$, the space of smooth maps from $M$ to $M$, at ${\rm id}_M$. Because ${\rm Diff}(M)$ is an open submanifold of $C^\infty(M,M)$ in the Whitney $C^\infty$-topology, $C^{\infty}(M,TM)$ is also the tangent space of ${\rm Diff}(M)$ at ${\rm id}_M$. Note that $C^\infty(M,TM)$ is just the space $\mathfrak{X}(M)$ of smooth vector fields on $M$. In addition, $\mathfrak{X}(M)$ is a Lie algebra under the Lie bracket operation, and hence the Lie algebra of ${\rm Diff}(M)$. Together with the naturality of Lie derivatives, that is, $\mathcal{L}_{[f,g]}=[\mathcal{L}_f,\mathcal{L}_g]$ for any $f,g\in\mathfrak{X}(M)$, the space of Lie derivatives is a Lie algebra isomorphic to $\mathfrak{X}(M)$. Since Lie derivatives are derivations on $C^\infty(M)$, we conclude that the Lie algebra of ${\rm Aut}(C^\infty(M))$ is exactly the space of Lie derivatives.
\end{proof}

Computationally, given a coordinate chart $(x_1,\dots,x_n)$ of $M$ defined on an open subset $U\subseteq M$, any vector field $f\in\mathfrak{X}(M)$ admits the local coordinate representation $f=\sum_{i=1}^nf_i\frac{\partial}{\partial x_i}$ for some $f_i\in C^\infty(U)$. Correspondingly, for any observable $h\in C^\infty(M)$ of the system in \eqref{eq:control_affine}, we have
\begin{align*}
\mathcal{L}_fh=dh(f)=\sum_{i=1}^nf_i\frac{\partial h}{\partial x_i}
\end{align*}
and hence the local coordinate representation of the Lie derivative operator $\mathcal{L}_f$ is identical to that of $f\in\mathfrak{X}(M)$ as 
\begin{align}
\label{eq:Lie_derivative}
\mathcal{L}_f=\sum_{i=1}^nf_i\frac{\partial}{\partial x_i}.
\end{align}
This in turn gives a coordinate-dependent argument for the isomorphism between the Lie algebras of smooth vectors fields and Lie derivatives on $M$. From a different perspective, we notice that any partial differential operator of homogenous order 1 on $M$ has the local representation in the form of \eqref{eq:Lie_derivative}. Consequently, the Lie algebra of ${\rm Aut}(C^\infty(M))$ can also be identified with the space of homogeneous first order partial differential operators on $M$. 

\begin{remark}[System-theoretic Gelfand duality]
As indicated in the proof of Lemma \ref{lem:Lie_algebra}, ${\rm Diff}(M)$ is related to ${\rm Aut}(C^\infty(M))$ by the Gelfand representation. On the dynamical systems level, the time-$t$ flow of the control-affine system in \eqref{eq:control_affine} is always an element of ${\rm Diff}(M)$ for any choice of the control inputs $u_i$, $i=1,\dots,m$ as defined in Section \ref{sec:flow}. Therefore, the associated Koopman system defined on ${\rm Aut}(C^\infty(M))$ in \eqref{eq:Koopman_control_affine} associated with it can be interpreted as the "Gelfrand representation" of the system in \eqref{eq:control_affine} on the space $C^\infty(M)$ of observables, which further reveals the Gelfand duality between control-affine systems and bilinear Koopman systems in the algebraic sense.
\end{remark}

\subsubsection{Controllability of Koopman systems}

Because controllability of a control-affine system is determined by the Lie algebra generated by its drift and control vector fields, the aforementioned identification of these vector fields with the corresponding Lie derivative operators, governing the dynamics of the associated Koopman system, opens up the possibility to study controllability of control-affine systems by utilizing their associated bilinear Koopman systems. In particular, the interpretation of Lie derivatives as partial differential operators gives rise to the follows operator-theoretic characterization of controllability for control-affine systems.

\begin{theorem}
\label{thm:Koopman_controllability_d}
A control-affine system defined on a closed manifold $M$ as in \eqref{eq:control_affine} is controllable if and only if the associated Koopman system as in \eqref{eq:Koopman_control_affine} acts on the space of observables $C^\infty(M)$ by the de Rham differential operator $d:C^\infty(M)\rightarrow C^\infty(M,T^*M)$, where $C^\infty(M,T^*M)$ denotes the space of smooth sections of the cotangent bundle $T^*M$ of $M$, i.e., the space of differential 1-forms on $M$.   
\end{theorem} 
\begin{proof}
According to the Lie algebra rank condition (LARC), the system in \eqref{eq:control_affine} is controllable on $M$ if and only if the Lie algebra generated by the vector fields $f$ and $g_i$, $i=1,\dots,m$, evaluated at each point $x\in M$, is the tangent space $T_xM$ of $M$ at the point $x$ \cite{Isidori95}. The isomorphism between the Lie algebras generated by smooth vector fields and Lie derivative operators revealed in the proof of Lemma \ref{lem:Lie_algebra} then implies that, for any vector field $v\in\mathfrak{X}(M)$, the tangent vector $v(x)\in T_xM$ belongs to the Lie algebra generated by $\mathcal{L}_f$ and $\mathcal{L}_{g_i}$, $i=1,\dots,m$, evaluated at $x$. Recall the action of Lie derivatives on smooth functions as $\mathcal{L}_vh(x)=dh(v)|_x=dh_x(v(x))$, the directional derivative of $h$ at $x$ along the direction of $v(x)$, the above observation implies that there exist control inputs $u_i(t)$, $i=1,\dots,m$ such that $\frac{d}{dt}U_th(x)=\mathcal{L}_{v}U_th(x)=d(U_th)_x(v(x))$. Because $v(x)\in T_xM$ is arbitrary, we obtain $\frac{d}{dt}U_th=d(U_th)$. In addition, note that $U_t:C^\infty(M)\rightarrow C^{\infty}(M)$ is invertible, which particularly implies the subjectivity of $U_t$ so that $U_th\in C^\infty(M)$ can be arbitrarily chosen as well. Therefore, we arrive at the conclusion $\frac{d}{dt}U_t=d$. 
\end{proof}

The de Rham differential is globally defined for any smooth manifold. Therefore, Theorem \ref{thm:Koopman_controllability_d} also gives rise to a global characterization of controllability for a control-affine system in terms of differential operators, contrary to the Lie algebra rank condition (LARC) which involves the local examination of the Lie algebra generated by the drift and control vector fields of the system at every point in the state-space manifold of the system. 

To generalize Theorem \ref{thm:Koopman_controllability_d} to uncontrollable systems, we notice that a system is always controllable on its controllable submanifold. In particular, for a control-affine system as in \eqref{eq:control_affine}, the Frobenius theorem further implies that the controllable submanifold is a maximal integral manifold of the involutive distribution, that is, the Lie subalgebra of the vector fields, generated by the drift and control vector fields of the system, and the collection of all controllable submanifolds of the system starting from different initial conditions forms a foliation of the state-space manifold. Then, the proof of Theorem \ref{thm:Koopman_controllability_d} can be directly adopted by restricting to the controllable submanifold of the system in \eqref{eq:control_affine}.

\begin{corollary}
\label{cor:Koopman_controllability_d}
Consider the control-affine system in \eqref{eq:control_affine} defined on the manifold $M$. Let $N\subseteq M$ be the controllable submanifold of the system and $\iota:N\hookrightarrow M$ denote the inclusion. Then, the associated Koopman system in \eqref{eq:Koopman_control_affine} acts on the space of observables $C^\infty(M)$ by $\iota^* d$, where $d$ denotes the de Rham differential operator on $M$ and $\iota^*$ is the pullback of $\iota$. 
\end{corollary}
\begin{proof}
The condition that $N$ is the controllable submanifold of the system in \eqref{eq:control_affine} implies that the Lie subalgebra of $\mathfrak{X}(M)$ generated by the vector fields $f$ and $g_i$, governing the dynamics of the system, coincides with $T_xN$ evaluated at each $x\in N$. As a result, for any $v\in\mathfrak{X}(N)$, there are control inputs $u_i$, $i=1,\dots,m$ such that $\frac{d}{dt}U_th(x)=\mathcal{L}_{v(x)}U_th(x)=d(U_th)_x(v(x))=d(U_th)_{\iota(x)}(\iota_*v(x))=\iota^*d(U_th)_x(v(x))$ by the definition of the pullback \cite{Lee12}, where $\iota_*$ denotes the pushforward of $\iota$. Because $h\in C^\infty(M)$ and $v\in\mathfrak{X}(N)$ are arbitrary, we conclude $\frac{d}{dt}U_t=\iota^*d$ as desired.
\end{proof}

On the other hand, the proof of Theorem \ref{thm:Koopman_controllability_d} further sheds light on the controllability analysis of Koopman systems through their actions on the observables of the corresponding control-affine systems. To put forward this idea, we now turn our attention to controllability of the Koopman system on ${\rm Aut}(C^\infty(M))$. However, it is so unfortunate that Koopman systems are almost never controllable on ${\rm Aut}(C^\infty(M))$, even though the corresponding control-affine systems controllable on $M$ as shown in the following example.
\begin{example}[Uncontrollability of Koopman systems]
consider the control-affine system
\begin{align*}
\frac{d}{dt}x(t)=\sum_{i=1}^nu_i(t)e_i(x(t))
\end{align*}
defined on an $n$-dimensional smooth manifold $M$, where $e_i\in\mathfrak{X}(M)$, $i=1,\dots,n$ form a global frame for $M$ restricting to the coordinate vector fields $e_i=\frac{\partial}{\partial x_i}$ in a coordinate chart. Then, the system is clearly controllable on $M$ by the LARC, since as a frame for $M$, they satisfy ${\rm span}\{e_1(p),\dots,e_n(p)\}=T_pM$ for all $p\in M$. However, because of $e_i=\frac{\partial}{\partial x_i}$, $[e_i,e_j]=0$, equivalently $[\mathcal{L}_{e_i},\mathcal{L}_{e_j}]=\mathcal{L}_{[e_i,e_j]}=0$, holds for all $i,j=1,\dots,n$. Correspondingly, the Lie algebra generated by $\mathcal{L}_{e_i}$, $i=1,\dots,n$ is just ${\rm span}\{\mathcal{L}_{e_1},\dots, \mathcal{L}_{e_n}\}$, which is a proper Lie subalgebra of $\mathfrak{X}(M)$, only consisting of ``constant vector fields", i.e, vector fields in the form of $f=\sum_{i=1}^nf_i\frac{\partial}{\partial x_i}$ with all of $f_i$ constant functions on $M$. In the terminology of algebra, it means that the Lie algebra generated by $\mathcal{L}_{e_i}$, $i=1,\dots,n$ fails to be a module over $C^\infty(M)$ so that Koopman operators generated by flows of ``nonconstant vector fields" on $M$ are not in the reachable set of the associated Koopman system
\begin{align}
\label{eq:Koopman_uncontrollable}
\frac{d}{dt}U_t=\sum_{i=1}^nu_i(t)\mathcal{L}_{e_i}U_t,
\end{align}
which in turn  implies uncontrollability of the Koopman system on ${\rm Aut}(C^\infty(M))$.
\end{example}

Elaborating uncontrollability of the Koopman system in \eqref{eq:Koopman_uncontrollable} from a purely control-theoretic perspective, the trouble that the Lie algebra generated by $\mathcal{L}_{e_i}$ fails to be a $C^\infty(M)$-module is caused by the independence of the control inputs $u_i$ from the state-space $M$ of the corresponding control-affine system, meaning, $u_i$ are function defined on $\mathbb{R}$ solely instead of $\mathbb{R}\times M$ as in the case of control of partial differential equation systems. However, this is definitely not the only obstacle to controllability of Koopman systems. Another one follows from the definition of Koopman operators as composition operators with flows of vector fields on $M$, which restricts the reachable set of the Koopman system to elements in ${\rm Aut}(C^\infty(M))$, equivalently ${\rm Diff}(M)$, generated by flows of vector fields. In general, not every diffeomorphism on $M$, even in arbitrary small neighborhood of the identity map ${\rm id}_M$ of $M$ in ${\rm Diff}(M)$ in the $C^\infty$-topology, comes from a flow of some vector field. The most notable example was provided by the American mathematician John Milnor in 1984 \cite{Milnor84}: The function $\Phi(\theta)=\theta+\frac{\pi}{n}+\varepsilon\sin^2(n\theta)$ is a diffeomorphism of the unit circle $\mathbb{S}^1$ for any $n\in\mathbb{N}$ and $\varepsilon>0$, and particularly, by choosing large enough $n$ and small enough $\varepsilon$, $f$ can be arbitrarily close to ${\rm id}_M$. The points in $\mathbb{S}^1$ of the form $\theta_k=k\pi/n$ with $k\in\mathbb{Z}$ are $2n$-periodic, i.e., $f^{2n}(\theta_k)=\theta_k$, where $f^{2n}$ denotes the $2n$-fold composition of $f$, but other points are all aperiodic. This fact prevents $f$ from being a flow of a vector field on $\mathbb{S}^1$, and further indicates that, in general, the exponential map from the Lie algebra $\mathfrak{X}(M)$ to the Lie group ${\rm Diff}(M)$ is not a local diffeomorphism, actually neither locally injective nor locally surjective, which demonstrates a significant difference between infinite-dimensional and finite-dimensional Lie groups. Moreover, this also disables the application of LARC to examine controllability of Koopman systems. 

Fortunately, by leveraging the algebraic structure of the Lie group ${\rm Diff}(M)$, it is still possible to characterize the controllable submanifold of the Koopman system defined on ${\rm Aut}(C^\infty(M))$

\begin{theorem}[Controllable submanifolds of Koopman systems]
\label{thm:Koopman_controllability}
Given a Koopman system defined on the Lie group ${\rm Aut}(C^\infty(M))$ as in \eqref{eq:Koopman_control_affine}, that is, 
$$\frac{d}{dt}U_t=\mathcal{L}_fU_t+\sum_{i=1}^mu_i(t)\mathcal{L}_{g_i}U_t$$
associated to a control-affine system defined on $M$ as in \eqref{eq:control_affine}, the controllable submanifold of the Koopman system is the identity component of the Lie subgroup of ${\rm Aut}(C^\infty(M))$ whose Lie algebra is the Lie subalgebra of $\mathfrak{X}(M)$ generated by $\mathcal{L}_f$ and $\mathcal{L}_{g_i}$, $i=1,\dots,n$.
\end{theorem}
\begin{proof}
Let $\mathfrak{g}$ denotes the Lie subalgebra of $\mathfrak{X}(M)$ generated by $\mathcal{L}_f$ and $\mathcal{L}_{g_i}$, $i=1,\dots,n$, and $G$ be the Lie subgroup of ${\rm Aut}(C^\infty(M))$ with the Lie algebra $\mathfrak{g}$. Then, by applying piecewise constant control inputs, it can be shown that elements in ${\rm Aut}(C^\infty(M))$ of the form $\exp(\mathcal{L}_{v_1})\cdots\exp(\mathcal{L}_{v_k})$ with $\mathcal{L}_{v_i}\in\mathfrak{g}$ for all $i=1,\dots,k$ are in the reachable set of the Koopman system in \eqref{eq:Koopman_control_affine}. Note that all such elements form a normal subgroup of $G_0$, the identity component of $G$, but $G_0$ is a simple group \cite{Thurston74}, i.e., does not contain any normal subgroups in addition to the trivial group and itself. Therefore, these elements constitute the whole $G_0$. Moreover, because every state in the reachable set of a control-affine system can be reached by applying piecewise constant control inputs \cite{Jurdjevic96}, we conclude that $G_0$ is the controllable submanifold of the Koopman system in \eqref{eq:Koopman_control_affine}.
\end{proof}

It is worth mentioning that the proof of Theorem \ref{thm:Koopman_controllability} completely relies on the algebraic structure of the Lie algebra generated by the Lie derivative operators $\mathcal{L}_f$ and $\mathcal{L}_{g_i}$, equivalently, the vector fields $f,g_i\in\mathfrak{X}(M)$, and more specifically, it is independent of whether the system in \eqref{eq:Koopman_control_affine} is a Koopman system or not. Therefore, Theorem  \ref{thm:Koopman_controllability} gives rise to a universal characterization of controllable submanifolds for bilinear systems defined on the Lie group ${\rm Aut}(C^\infty(M))$, which can be viewed as the infinite-dimensional analogue of the LARC for bilinear systems evolving on Lie groups. 
 
To present the result in the most general setting, we temporally relax the compactness condition on the manifold $M$. In this case, the Lie algebra of the Lie group ${\rm Diff}(M)$ becomes $\mathfrak{X}_c(M)$, the space of compactly support smooth vector fields on $M$. Technically, to guarantee that the map $\mathbb{R}\rightarrow{\rm Diff}(M)$ given by $t\mapsto\Phi_t$ is smooth for any flow $\Phi$ of a vector field in $\mathfrak{X}_c(M)$, we equip ${\rm Diff}(M)$ with the final topology generated by smooth curves from $\mathbb{R}$ to $C^\infty(M,M)$, the space of all smooth maps form $M$ to $M$. This topology is generally weaker than the Whitney $C^\infty$-topology, but coincides with it when $M$ is compact \cite{Kriegl97}. Algebraically, the isomorphism between ${\rm Aut}(C^\infty(M))$ and ${\rm Diff}(M)$, as well as their Lie algebras, still holds. Moreover, if $M$ is a manifold with boundary, then we also require vector fields in $\mathfrak{X}_c(M)$ to be vanishing on the boundary.

\begin{corollary}[Infinite-dimensional LARC]
Let $M$ be a smooth manifold, possibly noncompact and with boundary, and 
\begin{align*}
\frac{d}{dt}A(t)=\mathcal{L}_fA(t)+\sum_{i=1}^mu_i(t)\mathcal{L}_{g_i}A(t)
\end{align*}
be a bilinear system defined on the Lie group ${\rm Aut}(C^\infty(M))$, where $f,g_i\in\mathfrak{X}_c(M)$ for all $i=1,\dots,m$. Then, the controllable submanifold of the system is the identity component of the Lie subgroup of ${\rm Aut}(C^\infty(M))$ whose Lie algebra is the Lie subalgebra of $\mathfrak{X}_c(M)$ generated by $\mathcal{L}_f$ and $\mathcal{L}_{g_i}$, $i=1,\dots,m$. 
\end{corollary} 
\begin{proof}
The proof is identical to that of Theorem \ref{thm:Koopman_controllability}.
\end{proof}

\section{Koopman Feedback Linearization}
\label{sec:Koopman_feedback_linearization}
We have launched a detailed investigation into Koopman systems and their relationship with the dynamical systems that they are associated with. In particular, for control-affine systems, the bilinear form of the associated Koopman systems greatly benefits the study of the original systems in a coordinate-free way, e.g., the characterize of controllability in terms of the de Rham differential as shown in Theorem \ref{thm:Koopman_controllability_d}. In general, it is no doubt that linear systems are preferable from various aspects. Therefore, the focus of this section will be on the search of appropriate coordinates by using Koopman systems under which the corresponding control-affine systems can be represented as linear systems. In addition, we will also indicate the relation and distinction between the proposed linearization and the well-known feedback linearization. As the prerequisite, we will first explore how to change the coordinates of control-affine systems by using the associated Koopman systems. 

\subsection{Koopman change of coordinates}
\label{sec:Koopman_coordinates}

 Owing to the action of Koopman operators on the space of observables, for the purpose of changing the coordinate representation of a system, it is natural that the new coordinates must compose of the observables of the system. In addition, in the case of the state-space manifold $M$ of the system having dimension $n>1$, every coordinate chart of $M$ must contain $n$ components to guarantee $M$ to be locally homeomorphic to $\mathbb{R}^n$. Therefore, the use of Koopman operators to change the coordinates of the system requires the extension of the action of Koopman operators from $C^\infty(M)$ to $C^\infty(M,\mathbb{R}^n)$, the space of $\mathbb{R}^n$-valued smooth functions defined on $M$, where $\mathbb{R}^n$ is treated as a smooth manifold under its usual topology and smooth functions from $M$ to $\mathbb{R}^n$ are understood in the usual sense as smooth functions between two smooth manifolds. We also refer to such functions in $C^\infty(M,\mathbb{R}^n)$ as observables and their values in $\mathbb{R}^n$ as outputs. 
 
Note that $\mathbb{R}^n$, as a manifold, is parallelizable since the set of coordinate vector fields $\{\frac{\partial}{\partial x_1},\dots,\frac{\partial}{\partial x_1}\}$ restricts to a basis of $T_p\mathbb{R}^n$ at any $p\in\mathbb{R}^n$ \cite{Lee12}. As a direct consequence of this result, every elements $h\in C^\infty(M,\mathbb{R}^n)$ admits a representation as $n$-tuple $h=(h_1,\dots,h_n)'$ with each component $h_i\in C^\infty(M)$ and `$\prime$' denoting the transpose.  Equivalently, we reveal an isomorphism between $C^\infty(M,\mathbb{R}^n)$ and $\big(C^\infty(M)\big)^n$ as $C^\infty(M)$-algebras, where $\big(C^\infty(M)\big)^n$ denotes the direct sum of $n$ copies of $C^\infty(M)$, which in turn indicates the component-wise action of the Koopman operator $U_t$ on $C^\infty(M,\mathbb{R}^n)$ as 
$$U_th=U_t\left[\begin{array}{c} h_1 \\ \vdots \\ h_r \end{array}\right]=\left[\begin{array}{c} U_th_1 \\ \vdots \\ U_th_n \end{array}\right]\in C^\infty(M,\mathbb{R}^n)$$
so that $U_t\in{\rm Aut}(C^\infty(M,\mathbb{R}^n))$. In particular, in the case that the system is in the control-affine form and controllable on $M$, the component-wise action can then be extended from the Koopman operator $U_t$ to its infinitesimal generator $\frac{d}{dt}U_t$ as
$$\frac{d}{dt}U_th=\left[\begin{array}{c} \frac{d}{dt}U_th_1 \\ \vdots \\ \frac{d}{dt}U_th_r \end{array}\right]=\left[\begin{array}{c} d(U_th_1) \\ \vdots \\ d(U_th_n) \end{array}\right],$$
To elucidate the mechanism of $\frac{d}{dt}U_th$, we calculate its action on vector fields on $M$ by using local coordinates. Specifically, let $(x_1,\dots,x_n)$ and $(y_1,\dots,y_n)$ are coordinate charts on $M$ and $\mathbb{R}^n$, respectively, then $\frac{d}{dt}U_th$ has the local representation
\begin{align*}
\frac{d}{dt}U_th=\sum_{i=1}^nd(U_th_i)\otimes\frac{\partial}{\partial y_i}=\sum_{i=1}^n\sum_{j=1}^n\frac{\partial U_th_i}{\partial x_j}dx_j\otimes\frac{\partial}{\partial y_i},
\end{align*}
which exactly gives the Jacobian matrix, equivalently the coordinate representation of the differential, of the function $U_th$, where $\otimes$ denotes the tensor product of vector bundles, to be more specific, the tensor product between the cotangent bundle $T^*M$ of $M$ and the tangent bundle $T\mathbb{R}^n$ of $\mathbb{R}^n$. Let $f=\sum_{i=1}^nf_i\frac{\partial}{\partial x_i}$ be a smooth vector field on $M$, then the action of $\frac{d}{dt}U_t$ on $f$ is given by
\begin{align}
\frac{d}{dt}U_th(f)&=\sum_{i=1}^n\Big(\sum_{j=1}^n\frac{\partial U_th_i}{\partial x_j}dx_j\Big)\Big(f_k\frac{\partial}{\partial x_k}\Big)\otimes\frac{\partial}{\partial y_i}\nonumber\\
&=\sum_{i=1}^n\Big(\sum_{j=1}^nf_j\frac{\partial U_th_i}{\partial x_j}\Big)\frac{\partial}{\partial y_i},\label{eq:Koopman_pushforward}
\end{align}
where we use the duality between $T^*M$ and $TM$ as $dx_j(\frac{\partial}{\partial x_k})=\delta_{jk}$ with $\delta_{jk}=1$ if $j=k$ and $\delta_{jk}=0$ otherwise for all $j,k=1,\dots, n$. Note that the equation in \eqref{eq:Koopman_pushforward} coincides with $dU_th(f)$, the differential of $U_th$ acting on $f$ \cite{Lee12}, which then gives rise to the following conclusion.

\begin{proposition}
\label{prop:Koopman_pushforward}
Given a controllable control-affine system evolving on a smooth manifold $M$. Then, for any observable $h\in C^\infty(M,\mathbb{R}^n)$ of the system, the associated Koopman system acting on $h$ gives the pushforward of smooth vector fields on $M$ by the flow of the system output. 
\end{proposition}
\begin{proof}
This is largely an exercise in decoding terminologies. We first assume that the system is controllable on $M$. As shown in \eqref{eq:Koopman_pushforward}, for any $f\in\mathfrak{X}(M)$, $\frac{d}{dt}U_th(f)$ defines a section $d(U_th)(f)$ of the pullback bundle $(U_th)^*T\mathbb{R}^n$ over $M$, that is, the pushforward of $f$ by $U_th\in C^\infty(M,\mathbb{R}^n)$, denoted by $(U_th)_*f$ \cite{Jost17}. Furthermore, let $\Phi_t$ denote the time-$t$ flow of the system, then $U_th=h\circ\Phi_t$ exactly gives the follow of the output generated by the observable $h$. In another words, $(U_th)_*f=(h\circ\Phi_t)_*f$ is the pushforward of $f$ by the flow of the system output as desired. 

In the case that the controllable submanifold of the system is $N\subset M$, by Corollary \ref{cor:Koopman_controllability_d}, we have $\frac{d}{dt}U_th(f)=\iota^*d(U_th)(f)=d(\iota^*U_th)(f)$, where $\iota:N\hookrightarrow M$ is the inclusion of $N$ into $M$ and we use the commutativity between the de Rham differential and pullback maps. However, because the system is controllable only on $N$, the flow $\Phi_t$ of the system completely lies in $N$ so that the system output is given by $U_th=h\circ\Phi_t=h\circ\Phi_t\circ\iota=\iota^*(U_th)$. Therefore, the result follows the same proof as in the controllable case above.
\end{proof}

Specifically, for pursuing the goal of changing coordinates, it is of particular interest to apply the pushforward map $\frac{d}{dt}U_th=(U_th)_*$ to the vector field governing the dynamics of the system $\frac{d}{dt}x(t)=f(x(t))+\sum_{i=1}^mu(t)g(x(t))$. This yields
\begin{align}
&\frac{d}{dt}U_th\Big(f+\sum_{i=1}^mu_ig_i\Big)=(U_th)_*\Big(f+\sum_{i=1}^mu_ig_i\Big)\nonumber\\
&=(h\circ\Phi_t)_*\Big(f+\sum_{i=1}^mu_ig_i\Big)=h_*\Big[(\Phi_t)_*\Big(f+\sum_{i=1}^mu_ig_i\Big)\Big]\nonumber\\
&=h_*\Big[\Big(f+\sum_{i=1}^mu_ig_i\Big)\circ\Phi_t\Big]=h_*\Big(f+\sum_{i=1}^mu_ig_i\Big)\circ\Phi_t\label{eq:change_coordinates}
\end{align}
where $(\Phi_t)_*\big(f+\sum_{i=1}^mu_ig_i\big)=\big(f+\sum_{i=1}^mu_ig_i\big)\circ\Phi_t$ follows from the fact that the vector field $f+\sum_{i=1}^mu_ig_i$ is invariant under its own flow $\Phi_t$.

\begin{theorem}
\label{thm:change_coordinates}
Given any control-affine system, the action of the associated Koopman system on an observable that restricts to a local diffeomorphism on the controllable submanifold of the system defines a coordinate transformation for the system. 
\end{theorem}
\begin{proof}
We adopt the notations above that $M$ and $N$ denote the state-space manifold and controllable submanifold of the system, respectively, and $h\in C^\infty(M,\mathbb{R}^n)$ is the observable of the system satisfying the required assumption. Then, for any state $x(t)\in N$, there is a neighborhood $U$ in $N$ containing $x(t)$ diffeomorphic to $h(U)\subseteq\mathbb{R}^n$. On the other hand, because $N$ is the controllable submanifold of the system, the restrictions of the drift and control vector fields $f$ and $g_i$ to $N$ are tangent to $N$, i.e., are not only vector fields along $N$ but also vector fields on $N$. In the following, we identify $h$, $f$ and $g_i$ with their restrictions on $N$, where the system evloves. The diffeomorphism between $U$ and $h(U)$ then implies that $h_*f$ and $h_*g$, induced by the action of the associated Koopman system on $h$ as shown in \eqref{eq:change_coordinates}, are well-defined vector fields on $h(U)$, not only sections of the pullback bundle $h^*Th(U)$. This immediately enables the representation of the system in the coordinates of $h(U)$. Specifically, under the coordinate $y=h(x)$ on $h(U)$ induced by the observable $h\in C^\infty(M,\mathbb{R}^n)$, the flow of the system is exactly $h\circ\Phi_t=U_th$, and hence the derivation in \eqref{eq:change_coordinates} leads to the representation
\begin{align}
\label{eq:Koopman_coordinate}
\frac{d}{dt}y(t)=h_*f(y(t))+\sum_{i=1}^mu_i(t)h_*g_i(y(t))
\end{align}
of the system the system $\frac{d}{dt}x(t)=f(x(t))+\sum_{i=1}^mu(t)g(x(t))$ in this coordinate on $h(U)$, concluding the proof.
\end{proof}

It is worth to emphasize again that in the system in \eqref{eq:Koopman_coordinate}, $h$, $f$ and $g_i$ should be interpreted as their restrictions on the controllable submanifold where the original system $\frac{d}{dt}x(t)=f(x(t))+\sum_{i=1}^mu_i(t)g_i(x(t))$ evolves on. More importantly, this further illustrates a distinctive feature as well as a great advantage of the Koopman change of coordinates. Conventionally, an observable $h\in C^\infty(M,\mathbb{R}^n)$ is required to be a diffeomorphism to generate a change of coordinates for the system on the state-space manifold $M$. However, by leveraging the associated Koopman system, the change of coordinates for the system can be localized to a controllable submanifold $N$ with the requirement for $h$ relaxed to be a local diffeomorphism when restricted to $N$. This localization is exactly enabled by the composition with the system flow $\Phi_t$ in the Koopman operator, for $\Phi_t$ is defined on the space where the system evolves as revealed in the proof of Proposition \ref{prop:Koopman_pushforward}.

\subsection{Koopman feedback linearization}

It is no doubt that linear systems are generally preferable, as an application of the Koopman change of coordinates, this section focuses on the case that a control-affine system as in \eqref{eq:control_affine} can be transformed to a linear system in the coordinates induced by the associated Koopman system. In particular, recalling the transformed system in \eqref{eq:Koopman_coordinate}, if it is a linear system, then $h_*f$ is a linear vector field and $h_*g_i$, $i=1,\dots,m$ are all constant vector fields. This imposes extremely strong conditions on the original system, e.g., the pushforward of the Lie bracket $h_*[g_i,g_j]$ between any pair of control vector fields $g_i$ and $g_j$ has to vanish identically because the pushforward operation is natural in the sense of $h_*[g_i,g_j]=[h_*g_i,h_*g_i]=0$ \cite{Lee12}, where $h_*g_i$ and $h_*g_j$ are constant vector fields and hence have trivial Lie bracket. For the purpose of looking for less restrictive conditions, we first note that, in the transformed system in \eqref{eq:control_affine}, the new coordinates are actually constructed from the output function $h$ in the form of $y(t)=h(x(t))$, and this inspires the adoption of the feedback linearization technique. The idea of this technique is to apply feedback control laws in the form of $u_i(t)=\alpha_i(x)+\sum_{j=1}^m\beta_{ij}(x)v_j(t)$ for some $\alpha_i,\beta_{ij}\in C^\infty(M)$ and all $i,j=1,\dots,m$ so that the forced system 
\begin{align}
\frac{d}{dt}x(t)=\big[f(x(t))+\sum_{i=1}^m&\alpha_i(x(t))g_i(x(t))\big]\nonumber\\
&+\sum_{i=1}^m\sum_{j=1}^m\beta_{ij}v_j(t)g_i(x(t)) \label{eq:feedback}
\end{align} 
can be linearized by using a coordinate system determined by the observable $h$. However, it has been shown that feedback linearizability immediately implies controllability, i.e., controllability is a necessary condition to guarantee the system to be feedback linearizable \cite{Isidori95}. In the sequel, we will integrate the feedback linearization technique into the Koopman framework, and as a major advantage, the controllability condition can be dropped due to the fact that the change of coordinates generated by the associated Koopman system can be localized to the controllable submanifold of the system as discussed at the end of Section \ref{sec:Koopman_coordinates}.

Specifically, let $N$ be the controllable submanifold of the system and $\iota:N\hookrightarrow M$ the inclusion, then we assume that there is an observable $h=(h_1,\dots,h_l)\in C^\infty(M,\mathbb{R}^l)$ satisfying the following conditions in a neighborhood $U\subseteq N$ of a point $p\in N$:
\begin{itemize}
\item for each $i=1,\dots,l$ and $j=1,\dots,m$, there is some integer $r_i$ such that $\L_{g_i}\L^k_f(h_j\circ\iota)=0$ for all $k=1,\dots,r_i-2$.
\item the matrix of functions in $C^\infty(M)$
\begin{align*}
R=\left[\begin{array}{ccc} \L_{g_1}\L^{r_1-1}_f(h_1\circ\iota) & \cdots & \L_{g_m}\L^{r_1-1}_f(h_1\circ\iota) \\ \L_{g_1}\L^{r_2-1}_f(h_2\circ\iota) & \cdots & \L_{g_m}\L^{r_2-1}_f(h_2\circ\iota) \\ \vdots &  & \vdots \\ \L_{g_1}\L^{r_l-1}_f(h_l\circ\iota) & \cdots & \L_{g_m}\L^{r_l-1}_f(h_l\circ\iota)  \end{array}\right]
\end{align*}
has rank $l$. 
\end{itemize}
Then, we say that the system has \emph{relative degree} $(r_1,\dots,r_l)$ at $p$. 

\begin{proposition}
\label{prop:Koopman_independence}
Suppose the observable $h\in C^\infty(M,\mathbb{R}^l)$ of the control-affine system in \eqref{eq:control_affine} has the relative degree $(r_1,\dots,r_l)$ at $p\in N$, then the matrix of 1-forms on $M$
\begin{align}
\left[\begin{array}{cccc} \frac{d}{dt}U_th_1 & \frac{d}{dt}U_t\L_fh_1 & \cdots & \frac{d}{dt}U_t\L_f^rh_i \\ \frac{d}{dt}U_th_2 & \frac{d}{dt}U_t\L_fh_2 & \cdots & \frac{d}{dt}U_t\L_f^rh_2 \\ \vdots & \vdots & & \vdots \\\frac{d}{dt}U_th_l & \frac{d}{dt}U_t\L_fh_l & \cdots & \frac{d}{dt}U_t\L_f^rh_l \end{array}\right] \label{eq:Koopman_independence}
\end{align}
has rank $l$ evaluated at every points in a neighborhood $U$ of $p$ in $N$, where $U_t$ denote the associated Koopman operator and $r=\max\{r_1,\dots,r_l\}$. 
\end{proposition}
\begin{proof}
According to Corollary \ref{cor:Koopman_controllability_d} that $\frac{d}{dt}U_t$ acts on observables by $\iota^*d$, we have $\frac{d}{dt}U_t\L_f^kh_i=\iota^*d\L_f^kh_i=d\iota^*\L_f^kh_i=d(\L_f^kh_i\circ\iota)$, where we use the commutativity of the pullback and de Rham differential. We now claim $\L_fh\circ\iota=\L_f(h\circ\iota)$. To see this, we use the definition of Lie derivatives $\L_f(h\circ\iota)=d(h\circ\iota)(f)=dh\circ d\iota(f)$. Because $N$ is the controllable submanifold of the system, $f|_N\in\mathfrak{X}(N)$ holds, which implies $d\iota(f)=f|_N$, yielding the claim. By induction on the order $k$ of the successive Lie derivatives, we obtain $\frac{d}{dt}U_t\L_f^kh_i=d\L_f^k(h_i\circ\iota)$. 

We now consider the $l$-by-$m$ block matrix $G$ with the $(i,j)$-block $G_{ij}$ the $r_i$-by-$r$ matrix whose $(\alpha,\beta)$-entry is given by 
\begin{align*}
\frac{d}{dt}U_t\L^{\a-1}&h_i(\L_f^{\b-1}g_j)=d\L_f^{\a-1}(h_i\circ\iota)(\L_f^{\b-1}g_j)\\
&=d(\L_f^{\a-1}h_i)(\L_f^{\b-1}g_j)\circ\iota\\
&=\sum_{k=0}^{\b-1}(-1)^k{\b-1 \choose k}\L_f^{\b-1-k}\L_{g_j}\L_f^{\a-1+k}h_i\circ\iota
\end{align*}
where we use the facts $f|_N,g_i|_N\in\mathfrak{X}(N)$ and the product rule of Lie derivatives. Let $U\subseteq N$ be a neighborhood of $x$ as in the definition of relative degree, then in $U$ the matrix $G_{ij}$ satisfies
\begin{align*}
\frac{d}{dt}U_t\L^{\a-1}&h_i(\L_f^{\b-1}g_j)\\
&=\begin{cases}
0,\quad\text{if }\a+\b<r_i+1\\
(-1)^{\b-1}\L_{g_i}\l_f^{r_i-1}h_i\circ\iota,\quad\text{if }\a+\b=r_i+1
\end{cases}
\end{align*}
Correspondingly, the upper left triangular block of $G_{ij}$ consists of only 0, and the antidiagonal elements are equal to $(-1)^{\b-1}\L_{g_i}\l_f^{r_i-1}h_i$. Therefore, rearranging the rows of $G$ results in a block triangular matrix with each of the diagonal blocks equal to the matrix $R$. Because $A$ is full rank by the definition of the relative degree, the desired result follows. 
\end{proof}

An another interpretation of Proposition \ref{prop:Koopman_independence} is that the matrix in \eqref{eq:Koopman_independence} has full row rank, equivalently, the row vectors are linearly independent in $U\subseteq N$. This particularly implies that $(\frac{d}{dt}U_th_1,\dots,\frac{d}{dt}U_t\L_f^{r_1-1}h_1,\dots,\frac{d}{dt}U_th_l,\dots,\frac{d}{dt}U_t\L_f^{r_l-1}h_l)$ qualifies as a coordinate chart of $N$ on $U$ provided $r_1+\cdots+r_l={\rm dim}\,N$, which is exactly the desired coordinates linearizing the nonlinear control-affine system in \eqref{eq:feedback}, that is, the system in \eqref{eq:control_affine} steered by a feedback control law.

\begin{theorem}[Koopman feedback linearization]
\label{thm:Koopman_feedback_linear}
Given a control-affine system defined on $M$ as in \eqref{eq:control_affine}, and let $N\subseteq M$ be the controllable submanifold and $h=(h_1,\dots,h_l)\in C^\infty(M,\mathbb{R}^l)$ be an observable of the system. If the system has relative degree $(r_1,\dots,r_l)$ at $p\in N$ such that $r_1+\dots+r_l=n$, the dimension of $N$, then the action of the associated Koopman system on the observable $\phi=(h_1,\dots,\L_f^{r_1-1}h_1,\dots,h_l,\dots,\L_f^{r_i-1}h_l)\in C^\infty(M,\mathbb{R}^n)$ gives a linear system defined on and neighborhood $U\subseteq N$ of $p$ in the form of
\begin{align*}
\frac{d}{dt}U_t\phi=AU_t\phi+Bv(t)
\end{align*}
where 
$$A=\left[\begin{array}{ccc} A_1 & & \\ & \ddots & \\ & & A_l \end{array}\right]\in\mathbb{R}^{n\times n},\ B=\left[\begin{array}{ccc} b_1 & & \\ & \ddots & \\ & & b_l \end{array}\right]\in\mathbb{R}^{n\times m}$$
and $v(t)=(v_1(t),\dots,v_l(t))'\in\mathbb{R}^l$ with
$$A_i=\left[\begin{array}{ccccc} 0 & 1 & 0 & \cdots & 0 \\ 0 & 0 & 1 & \cdots & 0 \\ \vdots & \vdots & \vdots & \ddots  & \vdots \\ 0 & 0 & 0 & \cdots & 1 \\ 0 & 0 & 0 & \cdots & 0 \end{array}\right]\in\mathbb{R}^{r_i\times r_i},\ b_i=\left[\begin{array}{c} 0 \\ 0 \\ \vdots \\ 0 \\ 1 \end{array}\right]\in\mathbb{R}^{r_i}$$
and $v_i(t)=\sum_{j=1}^mu_j(t)U_t\L_{g_j}\L_f^{r_i-1}(h_i\circ\iota)+U_t\L_f^{r_i}(h_i\circ\iota)$, respectively, for all $i=1,\dots,l$.
\end{theorem}
\begin{proof}
We first recall that the action of the associated Koopman system $\frac{d}{dt}U_t\phi$ gives the differential of $U_t\phi$, then by Proposition \ref{prop:Koopman_independence}, the assumption $r_1+\dots+r_l=n$ implies that $\phi$ restricts to a diffeomorphism on a neighborhood $U\subseteq N$ of $x$ to $\mathbb{R}^n$. This implies that $\phi$ gives rise to a coordinate chart of $N$ on $U$. Next, to compute the coordinate representation of the system in \eqref{eq:control_affine} in this chart, pick any $q\in U$ as the initial condition of the system $x(0)=q$ and $t$ small enough so that the trajectory of the system stays in $U$, then for any $i=1,\dots,l$ and $k=0,\dots,r_i-1$, we have $U_th_i(q)=h(x(t))$ and
\begin{align*}
\frac{d}{dt}U_t\L_f^kh_i(p)&=\frac{d}{dt}\L_f^kh_i(x(t))=d\L_f^kh_i\Big(\frac{d}{dt}x(t)\Big)\\
&=d\L_f^kh_i\big(f+\sum_{j=1}^mu_jg_j\big)(x(t))\\
&=\L_f^{k+1}h_i(x(t))+\sum_{j=1}^mu_j\L_{g_j}\L_f^{k}h_i(x(t)),
\end{align*} 
Because $N$ is the controllable submanifold such that $x(t)\in U\subseteq N$ can be arbitrary, we obtain
\begin{align*}
\frac{d}{dt}U_t\L_f^kh_i\circ\iota&=U_t\L_f^{k+1}h_i\circ\iota+\sum_{j=1}^mu_jU_t\L_{g_j}\L_f^{k}h_i\circ\iota,
\end{align*} 
in which $\L_{g_j}\L_f^{k}h_i\circ\iota=0$ unless $k=r_i-1$ by the relative degree assumption and $\iota:N\hookrightarrow M$ denotes the inclusion. With the choice of $v(t)$ in the theorem statement, the desired result follows. 
\end{proof}

In the classical nonlinear systems theory, the notion of relative degree, and hence feedback linearization, is defined on a neighborhood of a point in the state-space of a system. Moreover, a feedback linearizable system is necessarily controllable, otherwise, the system can at most be partially linearized \cite{Isidori95}. However, the integration of the associated Koopman system localizes the feedback linearization technique to the controllable submanifold of the system by effectively eliminating the uncontrollable states of the system, which in turn greatly expands the scope of feedback linearization. 

\section{Conclusion}

In this paper, we develop a Koopman control framework for analyzing control systems by using Koopman operators. Based on the detailed investigation into the algebraic and analytical properties of Koopman operators, instead of spectral methods and ergodicity assumptions, we rigorously derive the infinite-dimensional differential equation system governing the dynamics of the Koopman operator, referred to as the Koopman system, associated with a control system. Especially, for a control-affine system, the associated Koopman system gives rise to a global bilinearizaton of the system. In addition, we further reveal the Gelfand duality between control-affine and bilinear Koopman systems, which greatly benefits controllability analysis of these systems. In part, it leads to the characterization of controllability for control-affine systems in terms of de Rham differenal operators. On the other hand, by leveraging techniques in infinite-dimensional geometry, we show that bilinear Koopman systems are defined on infinite-dimensional Lie groups and carry out an extension of LARC to such infinite-dimensional systems. The developed framework is then adopted in the context of feedback linearization, with the emphasis on the localization nature of Koopman systems. As a consequence, the scope of the classical feedback linearization technique is extended by including uncontrollable systems.  This work not only greatly expands the repertoire of tools in geometric control theory, but also sheds light on novel finite-dimensional approaches to the study of infinite-control dimensional control systems, through the bridge of the Gelfrand duality between finite-dimensional systems and the associated infinite-dimensional Koopman systems.

\bibliographystyle{ieeetr}
\bibliography{Koopman}

\end{document}